\documentclass[11pt,centertags]{amsart}
\usepackage{amscd}
\usepackage{easybmat}
\usepackage{mathrsfs}
\usepackage{amsfonts}
\usepackage{color}
\usepackage{pifont}
\usepackage{upgreek}
\usepackage{bm}
\usepackage{hyperref}
\usepackage{shorttoc}
\usepackage{amsmath,amstext,amsthm,a4,amssymb,amscd}
\usepackage[mathscr]{eucal}
\usepackage{mathrsfs}
\usepackage{epsf}
\numberwithin{equation}{section}

\def\p{\partial}
\def\o{\overline}
\def\b{\bar}
\def\mb{\mathbb}
\def\mc{\mathcal}
\def\n{\nabla}

\newtheorem{thm}{Theorem}[section]
\newtheorem{lemma}[thm]{Lemma}
\newtheorem{prop}[thm]{Proposition}
\newtheorem{cor}[thm]{Corollary}
\theoremstyle{definition}
\newtheorem{rem}[thm]{Remark}

\theoremstyle{definition}
\newtheorem{defn}[thm]{Definition}
\newcommand{\comment}[1]{}
\usepackage{fancyhdr}
\begin{document}

\title{Geodesic-Einstein metrics and nonlinear stabilities}

\author[Huitao Feng]{Huitao Feng$^1$}
\author[Kefeng Liu]{Kefeng Liu$^2$}
\author[Xueyuan Wan]{Xueyuan Wan}
\address{Huitao Feng: Chern Institute of Mathematics \& LPMC,
Nankai University, Tianjin, China}
\email{fht@nankai.edu.cn}

\address{Kefeng Liu, Department of Mathematics,
Capital Normal University, Beijing, 100048, China;
Department of Mathematics, University of California at Los Angeles, California 90095, USA}
\email{liu@math.ucla.edu}

\address{Xueyuan Wan: Mathematical Sciences, Chalmers University of Technology, University of Gothenburg, 41296 Gothenburg, Sweden}
\email{xwan@chalmers.se}

\thanks{$^1$~Partially supported by NSFC(Grant No. 11221091, 11271062, 11571184) and the Fundamental Research Funds for the Central Universities}
\thanks{$^2$~Partially supported by NSF (Grant No. 1510216);}

\begin{abstract}
In this paper, we introduce notions of nonlinear stabilities for a relative ample line bundle over a holomorphic fibration and define the notion of a geodesic-Einstein metric on this line bundle, which generalize the classical stabilities and Hermitian-Einstein metrics of holomorphic vector bundles. We introduce a Donaldson type functional and show that this functional attains its absolute minimum at geodesic-Einstein metrics, and we also discuss the relations between the existence of geodesic-Einstein metrics and the nonlinear stabilities of the line bundle. 
 As an application, we will prove that a holomorphic vector bundle admits a Finsler-Einstein metric if and only if it admits a  Hermitian-Einstein metric, which answers a problem posed by S. Kobayashi.
 \end{abstract}
\maketitle
\tableofcontents


\section*{Introduction} \label{s0}

In this paper, we study the triple $({\mc X},M,L)$, where $\pi:\mc{X}\to M$ is a holomorphic fibration with $\dim M=m$ and $\dim\mc{X}=m+n$, which means that $\pi:\mc{X}\to M$ is a proper surjective holomorphic mapping between complex manifolds $\mc{X}$ and $M$ whose differential has maximal rank everywhere such that every fiber is a compact complex manifold, and $L\to \mc{X}$ is a relative ample line bundle over $\mc X$, i.e. there exists a metric $\phi$ (more precisely, $e^{-\phi}$ is a metric)  on $L$ such that
$\sqrt{-1}\p\b{\p}_{\mc{X}/M}\phi>0$ fiberwise. Such a metric $\phi$ is called an admissible metric on $L$. We always assume in this paper that $\mc X$ is compact and $M$ is a compact K\"{a}hler manifold with a fixed K\"{a}hler form $\omega$.

For any admissible metric $\phi$ on $L$, the geodesic curvature $c(\phi)$ of $\phi$, which is a horizontal $(1,1)$ form on $\mc X$, is defined by (cf. \cite{Choi}, Definition 2.1):
\begin{align}\label{phi}
  c(\phi)=\left(\phi_{\alpha\b{\beta}}-\phi_{\alpha\b{l}}\phi^{k\b{l}}\phi_{k\b{\beta}}\right)\sqrt{-1}dz^{\alpha}\wedge d\b{z}^{\beta},
\end{align}
here the notations $\phi_{\alpha\b{\beta}}$, $\phi_{\alpha\b{l}}$ and $\phi^{k\b{l}}$ are defined in the Section \ref{section1}.

The geodesic curvature $c(\phi)$ plays an important role in many aspects (see, e.g.  \cite{Berm}, \cite{Choi}, \cite{Sch}, \cite{Wang}).
For the case of canonically polarized family, i.e., each fiber with $c_1<0$, the unique K\"ahler-Einstein metric on each fiber defines a  metric $\phi$ on the relative canonical bundle $K_{\mc{X}/M}$. By proving the positivity of geodesic curvature $c(\phi)$ along any curve, Schumacher \cite{Sch} proved that $K_{\mc{X}/M}$ is a positive line bundle if the family is nowhere infiniesimally trivial. 

When $\dim M=1$, the equation $c(\phi)=0$ is equivalent to the famous  homogenous complex Monge-Amp\`{e}re equation
$(\sqrt{-1}\p\b{\p}\phi)^{n+1}=0$ (cf. \cite{Don3}, \cite{Semmes}), which plays a crucial role in a lot of related important problems.
Note that in this case, the equation $c(\phi)=0$ can be also written as $tr_{\omega}c(\phi)=0$ for any metric $\omega$ on $M$. Inspired by this, for a general holomorphic fibration $\pi:\mc{X}\to M$ over a compact K\"{a}hler manifold $(M,\omega=\sqrt{-1}g_{\alpha\b{\beta}}dz^{\alpha}\wedge d\b{z}^{\beta})$, we introduce the notion of a geodesic-Einstein metric on $L$ with respect to $\omega$. We say that an admissible metric $\phi$ on $L$ is geodesic-Einstein with respect to $\omega$ if
\begin{align}
tr_{\omega}c(\phi)=g^{\alpha\b{\beta}}\left(\phi_{\alpha\b{\beta}}-\phi_{\alpha\b{l}}\phi^{k\b{l}}\phi_{k\b{\beta}}\right)=\lambda,
\end{align}
 where $\lambda$ is a constant.

In this paper, we mainly study the relations between the existence of geodesic-Einstein metrics on $L$ and certain notions of stability of $L$.

In Section 1, we first introduce the following Donaldson type functional $\mc{L}$ on the space $F^+(L)$ of admissible metrics on $L$: for any fixed $\psi\in F^+(L)$ and any $\phi\in F^+(L)$
 \begin{align}\label{F}
    \mc{L}(\phi,\psi)=\int_{M}\left(\frac{\lambda}{m}\mc{E}(\phi,\psi)\wedge\omega-\frac{1}{n+1}\mc{E}_1(\phi,\psi)\right)\frac{\omega^{m-1}}{(m-1)!},
  \end{align}
where $\mc{E}$ and $\mc{E}_1$ are given by (\ref{E0}) and (\ref{E1}) in this paper. This functional can be viewed as a generalization  of the famous Donaldson functional in the family case. By computing the first variation of the functional (\ref{F}), we can show that the critical points of this functional $\mc{L}$ coincide with the geodesic-Einstein metrics on $L$ (see Proposition \ref{prop2}). Moreover, by using X. Chen's geodesic approximation lemma (cf. \cite{Chen1}, Lemma 7; also \cite{FLW1}, Lemma 2.3), we get the following theorem:
\begin{thm}
  The  functional $\mc{L}(\cdot,\psi)$ attains its absolute minimum at the geodesic-Einstein metrics on $L$.
\end{thm}

The famous Donaldson-Uhlenbeck-Yau theorem reveals the deep relationship between the stability of a holomorphic vector bundle and the existence of Hermitian-Einstein metrics (cf. \cite{Nara}, \cite{Don0}, \cite{Don1}, \cite{Don2}, \cite{Yau}). In Section 2, we introduce the notions of the nonlinear semistability (stability) and the nonlinear polystablity associated to a triple $({\mc X},M,L)$ and discuss the relationships between the existence of geodesic-Einstein metrics on $L$ and these stabilities. We have
\begin{thm}\label{T} For a triple $(\mc{X},M,L)$, if $L$ admits a geodesic-Einstein metric, then the triple $(\mc{X},M,L)$ is nonlinear semistable and nonlinear polystable.
\end{thm}

To get a full understanding of the relationships among the notions of the geodesic-Einstein metric, nonlinear semistability and the nonlinear polystability would be an interesting problem.
For example, we could ask whether there exists a geodesic-Einstein metric for a nonlinear polystable triple $(\mc{X},M,L)$.

As an application, in Section 3, we study the special triple $(P(E),M,\mc{O}_{P(E)}(1))$ associated to a holomorphic vector bundle $E\to M$. By the Kobayashi correspondence (cf. \cite{Ko1}, \cite{FLW}, \cite{FLW1}), a Finsler metric $G$ on $E$ induces a natural admissible metric on $\mc{O}_{P(E)}(1)$. In this case, we can prove that the induced metric on $\mc{O}_{P(E)}(1)$ is geodesic-Einstein if and only if $G$ is Finsler-Einstein. So for a Finsler-Einstein vector bundle $E\to M$, we know from Theorem \ref{T} that the associated triple $(P(E),M,\mc{O}_{P(E)}(1))$ is nonlinear semistable and  nonlinear polystable. Also recall that a Finsler-Einstein vector bundle $E\to M$ is semistable (cf. \cite{FLW1}). Here a natural question is whether
a Finsler-Einstein vector bundle admits a Hermitian-Einstein metric. 

 In \cite{Ko4} S. Kobayashi extended the concept of the Hermitian-Einstein metric to the setting of complex Finsler geometry, he introduced the definition of a Finsler-Einstein metric on a holomorphic vector bundle. Furthermore, S. Kobayashi raised in \cite{Ko4} the following open problem:
 
{\bf Problem:} \textsl{What are the algebraic geometric consequences of the Finsler-Einstein condition? The first question in this regard is whether every Einstein-Finsler vector bundle is semi-stable or not?}

By using the Berndtsson's construction of the $L^2$-metric on the direct image bundle
$E=\pi_*(K_{\mc{X}/M}+L)$ (cf. \cite{Bern2}), we can answer the problem completely. 
\begin{thm}
  $E$ admits a Finsler-Einstein metric if and only if $E$ admits a Hermitian-Einstein metric. Therefore, the existence of Finsler-Einstein metrics is equivalent to polystable of the holomorphic vector bundle.
\end{thm}
\textbf{Acknowledgement}:
The third author would like to express his
gratitude to Professor Bo Berndtsson and Professor Robert Berman 
for numerous helpful discussions about this paper. The authors would like to thank the anonymous referee for valuable comments which helped to improve the paper.

\section{Geodesic-Einstein metrics and a Donaldson type functional}

In this section, we first introduce the notion of a geodesic-Einstein metric on $L$, and then introduce a Donaldson type functional on $F^+(L)$ and prove that this functional attains its absolute minimum at the geodesic-Einstein metrics on $L$.

\subsection{Geodesic-Einstein metrics}\label{section1} Let $\pi:\mc{X}\to M$ be a holomorphic fibration with compact fibres. Let $L$ be a relative ample line bundle over $\mc{X}$. As usual, we denote by
$(z;v)=(z^1,\cdots, z^m; v^1,\cdots, v^n)$ a local admissible holomorphic coordinate system of $\mc{X}$ with $\pi(z;v)=z$.

For any smooth function $\phi$ on $\mc{X}$, we denote
$$\phi_{\alpha}:=\frac{\p \phi}{\p z^{\alpha}},\quad \phi_{\b{\beta}}:=\frac{\p \phi}{\p \b{z}^{\beta}},
\quad \phi_{i}:=\frac{\p \phi}{\p v^i},\quad \phi_{\b{j}}:=\frac{\p \phi}{\p \b{v}^j},$$
where $1\leq i,j\leq n, 1\leq \alpha,\beta\leq m$.

Set
\begin{align*}
  F^+(L):=\{\phi | \phi \,\,\text{is an admissible metric on}\, L\}.
\end{align*}
For any $\phi\in F^+(L)$, set
\begin{align}\label{horizontal}
  \frac{\delta}{\delta z^{\alpha}}:=\frac{\p}{\p z^{\alpha}}-\phi_{\alpha\b{j}}\phi^{\b{j}k}\frac{\p}{\p v^k}.
\end{align}
By a routine computation, one can show that $\{\frac{\delta}{\delta z^{\alpha}}\}_{1\leq \alpha\leq m}$ spans a well-defined horizontal subbundle of $T\mc{X}$. In fact, for any two local admissible coordinate neighborhoods $\{U_{A}, (z_A,v_A)\}$ and $\{U_{B}, (z_B,v_B)\}$, if $U_A\cap U_B\neq \emptyset$, one has the following holomorphic functions:
\begin{align}
z_B=z_B(z_A),\quad v_B=v_B(z_A,v_A).	
\end{align}
So 
\begin{align}\label{tran1}
\frac{\p^2\phi}{\p v^i_A \p\b{v}^j_A}=\frac{\p^2\phi}{\p v^k_B \p\b{v}^l_B}\frac{\p v^k_B}{\p v^i_A}\o{\frac{\p v^l_B}{\p v^j_A}}	
\end{align}
and 
\begin{align}\label{tran2}
\frac{\p^2\phi}{\p z^{\alpha}_A\p\b{v}^j_A}=\frac{\p}{\p z^{\alpha}_A}\left(\frac{\p\phi}{\p\b{v}^l_B}\o{\frac{\p v^l_B}{\p v^j_A}}\right)	=\left(\frac{\p z^{\gamma}_B}{\p z^{\alpha}_A}\frac{\p^2\phi}{\p\b{v}^l_B\p z^{\gamma}_B}+\frac{\p v^k_B}{\p z^{\alpha}_A}\frac{\p^2\phi}{\p v^k_B\p\b{v}^l_B}\right)\o{\frac{\p v^l_B}{\p v^j_A}}.
\end{align}

From (\ref{tran1}) and (\ref{tran2}), one has
\begin{align}
\begin{split}
	&\quad \frac{\delta}{\delta z^{\alpha}_A}=\frac{\p}{\p z^{\alpha}_A}-\frac{\p^2\phi}{\p z^{\alpha}_A\p \b{v}^j_A}\left(\frac{\p^2\phi}{\p v^i_A\p\b{v}^j_A}\right)^{-1}\frac{\p}{\p v^i_A}\\
	&=\frac{\p z^{\gamma}_B}{\p z^{\alpha}_A}\frac{\p}{\p z^{\gamma}_B}+\frac{\p v^k_B}{\p z^{\alpha}_A}\frac{\p}{\p v^k_B}-\left(\frac{\p z^{\gamma}_B}{\p z^{\alpha}_A}\frac{\p^2\phi}{\p\b{v}^l_B\p z^{\gamma}_B}+\frac{\p v^k_B}{\p z^{\alpha}_A}\frac{\p^2\phi}{\p v^k_B\p\b{v}^l_B}\right)\o{\frac{\p v^l_B}{\p v^j_A}}\left(\frac{\p^2\phi}{\p v^i_A\p\b{v}^j_A}\right)^{-1}\frac{\p v^s_B}{\p v^i_A}\frac{\p}{\p v^s_B}\\
	&=\frac{\p z^{\gamma}_B}{\p z^{\alpha}_A}\frac{\p}{\p z^{\gamma}_B}+\frac{\p v^k_B}{\p z^{\alpha}_A}\frac{\p}{\p v^k_B}-\left(\frac{\p z^{\gamma}_B}{\p z^{\alpha}_A}\frac{\p^2\phi}{\p\b{v}^l_B\p z^{\gamma}_B}+\frac{\p v^k_B}{\p z^{\alpha}_A}\frac{\p^2\phi}{\p v^k_B\p\b{v}^l_B}\right)\left(\frac{\p^2\phi}{\p v^s_B \p\b{v}^l_B}\right)^{-1}\frac{\p}{\p v^s_B}\\
	&=\frac{\p z^{\gamma}_B}{\p z^{\alpha}_A}\left(\frac{\p}{\p z^{\gamma}_B}-\frac{\p^2\phi}{\p\b{v}^l_B\p z^{\gamma}_B}\left(\frac{\p^2\phi}{\p v^s_B \p\b{v}^l_B}\right)^{-1}\frac{\p}{\p v^s_B}\right)=\frac{\p z^{\gamma}_B}{\p z^{\alpha}_A}\frac{\delta}{\delta z^\gamma_B}.
\end{split}	
\end{align}
Therefore, $\{\frac{\delta}{\delta z^{\alpha}}\}_{1\leq \alpha\leq m}$ spans a subbundle of $T\mc{X}$.

Let $\{dz^{\alpha};\delta v^k\}$
denote the dual frame of $\left\{\frac{\delta}{\delta z^{\alpha}}; \frac{\p}{\p v^i}\right\}$. One has
$$\delta v^k=dv^k+\phi^{k\b{l}}\phi_{\b{l}\alpha}dz^{\alpha}.$$
Moreover, the differential operators
\begin{align}\label{HV}
\p^V=\frac{\p}{\p v^i}\otimes \delta v^i,\quad \p^H=\frac{\delta}{\delta z^{\alpha}}\otimes dz^{\alpha}.
\end{align}
are well-defined.

For any $\phi\in F^+(L)$, the geodesic curvature $c(\phi)$ of $\phi$ (cf. \cite{Choi}, Definition 2.1) is defined by
\begin{align}\label{cphi}
  c(\phi)=\left(\phi_{\alpha\b{\beta}}-\phi_{\alpha\b{j}}\phi^{i\b{j}}\phi_{i\b{\beta}}\right)\sqrt{-1} dz^{\alpha}\wedge d\b{z}^{\beta},
\end{align}
which is clearly a horizontal real $(1,1)$ form on $\mc X$. From the following lemma, one sees that the geodesic curvature $c(\phi)$ of $\phi$ is also well-defined.
\begin{lemma}\label{lemma1} The following decomposition holds,
  \begin{align}
    \sqrt{-1}\p\b{\p}\phi=c(\phi)+\sqrt{-1}\phi_{i\b{j}}\delta v^i\wedge \delta \b{v}^j.
  \end{align}
\end{lemma}
\begin{proof}
  By a direct computation, one has
  \begin{align*}
    &c(\phi)+\sqrt{-1}\phi_{i\b{j}}\delta v^i\wedge \delta \b{v}^j=\sqrt{-1}(\phi_{\alpha\b{\beta}}-\phi_{\alpha\b{l}}\phi^{k\b{l}}\phi_{k\b{\beta}})dz^{\alpha}\wedge d\b{z}^{\beta}\\
    &\quad+\sqrt{-1}\phi_{i\b{j}}(dv^i+\phi^{i\b{l}}\phi_{\b{l}\alpha}dz^{\alpha})\wedge (d\b{v}^j+\phi^{\b{j}k}\phi_{k\b{\beta}}d\b{z}^{\beta})\\
    &=\sqrt{-1}(\phi_{\alpha\b{\beta}}dz^{\alpha}\wedge d\b{z}^{\beta}+\phi_{\alpha\b{j}}dz^{\alpha}\wedge d\b{v}^j+\phi_{i\b{\beta}}dv^{i}\wedge d\b{z}^{\beta}+\phi_{i\b{j}}dv^i\wedge d\b{v}^j)\\
    &=\sqrt{-1}\p\b{\p}\phi.
  \end{align*}
\end{proof}

\begin{defn}\label{GeoEin} Let $\omega=\sqrt{-1}g_{\alpha\b{\beta}}dz^{\alpha}\wedge d\b{z}^{\beta}$ be a (fixed) K\"{a}hler metric on $M$.
  A metric $\phi\in F^+(L)$ is called a geodesic-Einstein metric on $L$ with respect to $\omega$ if it satisfies that
  \begin{align}\label{FE}
    tr_{\omega}c(\phi):=g^{\alpha\b{\beta}}\left(\phi_{\alpha\b{\beta}}-\phi_{\alpha\b{j}}\phi^{i\b{j}}\phi_{i\b{\beta}}\right)=\lambda,
  \end{align}
  where $\lambda$ is a constant.
\end{defn}

 \begin{prop}\label{prop1}
   Let $\phi$ be a geodesic-Einstein metric on $L$. Then
   \begin{align}\label{lamda}
     \lambda=\frac{2\pi m}{n+1}\frac{([\omega]^{m-1}c_1(L)^{n+1})[\mc{X}]}{([\omega]^mc_1(L)^n)[\mc{X}]}.
   \end{align}
   So $\lambda$ is a topological quantity depending only on the classes $c_1(L)$ and $[\omega]$.
 \end{prop}
\begin{proof} Since $\phi\in F^+(L)$ is a geodesic-Einstein metric, one has
  \begin{align}\label{sec1.1}
    \begin{split}
      tr_{\omega}c(\phi)\omega^m\wedge (\sqrt{-1}\p\b{\p}\phi)_n&=m c(\phi)\wedge \omega^{m-1}\wedge (\sqrt{-1}\p\b{\p}\phi)_n\\
      &=m\omega^{m-1}\wedge (\sqrt{-1}\p\b{\p}\phi)_{n+1},
    \end{split}
  \end{align}
  where $(\bullet)_r:=(\bullet)^r/r!$. Taking integral over $\mc{X}$ to the both sides of (\ref{sec1.1}), one gets
  \begin{align*}
    \lambda=\frac{m\int_{\mc{X}}\omega^{m-1}\wedge(\sqrt{-1}\p\b{\p}\phi)_{n+1}}{\int_{\mc{X}}\omega^m\wedge (\sqrt{-1}\p\b{\p}\phi)_n}=\frac{2\pi m}{n+1}\frac{([\omega]^{m-1}c_1(L)^{n+1})[\mc{X}]}{([\omega]^mc_1(L)^n)[\mc{X}]},
  \end{align*}
  which depends only on the classes $[\omega]$ and $c_1(L)$.
\end{proof}

\subsection{A Donaldson type functional} For any fixed metric $\psi\in F^+(L)$ on $L$ and any $\phi\in F^+(L)$, we define the following two functionals $\mc{E}, \mc{E}_1$:
 \begin{align}\label{E0}
   \mc{E}(\phi,\psi)=\frac{1}{n+1}\int_{\mc{X}/M}(\phi-\psi)\sum_{k=0}^n(\sqrt{-1}\p\b{\p}\phi)^k\wedge(\sqrt{-1}\p\b{\p}\psi)^{n-k}
 \end{align}
and
\begin{align}\label{E1}
  \mc{E}_1(\phi,\psi) =\frac{1}{n+2}\int_{\mc{X}/M}(\phi-\psi)\sum_{k=0}^{n+1}(\sqrt{-1}\p\b{\p}\phi)^k\wedge(\sqrt{-1}\p\b{\p}\psi)^{n+1-k}.
\end{align}
Note that $\mc{E}(\phi,\psi)$ is a smooth function, while $\mc{E}_1(\phi,\psi)$ is a smooth real $(1,1)$-form on $M$.

Now we introduce the following Donaldson type functional $\mc{L}$ on $F^+(L)$ by defining
\begin{align}\label{D}
    \mc{L}(\phi,\psi)=\int_{M}\left(\frac{\lambda}{m}\mc{E}(\phi,\psi)\wedge\omega-\frac{1}{n+1}\mc{E}_1(\phi,\psi)\right)\frac{\omega^{m-1}}{(m-1)!},
  \end{align}
where $\lambda$ is the constant given by (\ref{lamda}).

 Given any smooth family $\phi_t$ of admissible metrics on $L$, one has
 \begin{align*}
   \begin{split}
   \frac{d}{dt}\mc{E}(\phi_t,\psi)&=\frac{1}{n+1}\int_{\mc{X}/M}\dot{\phi}_t\sum_{k=0}^n(\sqrt{-1}\p\b{\p}\phi_t)^k\wedge(\sqrt{-1}\p\b{\p}\psi)^{n-k}\\
     &\quad+\frac{1}{n+1}\int_{\mc{X}/M}(\phi_t-\psi)\sum_{k=0}^n k(\sqrt{-1}\p\b{\p}\phi_t)^{k-1}\wedge(\sqrt{-1}\p\b{\p}\psi)^{n-k}\wedge \sqrt{-1}\p\b{\p}\dot{\phi}_t\\
     &=\frac{1}{n+1}\int_{\mc{X}/M}\dot{\phi}_t\sum_{k=0}^n(\sqrt{-1}\p\b{\p}\phi_t)^k\wedge(\sqrt{-1}\p\b{\p}\psi)^{n-k}\\
     &\quad+\frac{1}{n+1}\int_{\mc{X}/M}\dot{\phi}_t\sum_{k=0}^n k(\sqrt{-1}\p\b{\p}\phi_t)^k\wedge(\sqrt{-1}\p\b{\p}\psi)^{n-k}\\
     &\quad-\frac{1}{n+1}\int_{\mc{X}/M}\dot{\phi}_t\sum_{k=0}^n k(\sqrt{-1}\p\b{\p}\phi_t)^{k-1}\wedge(\sqrt{-1}\p\b{\p}\psi)^{n-k+1}\\
     &=\int_{\mc{X}/M}\dot{\phi}_t(\sqrt{-1}\p\b{\p}\phi_t)^n.
   \end{split}
 \end{align*}
Similarly, one has
\begin{align*}
  \frac{d}{d t}\mc{E}_1(\phi_t,\psi)\equiv\int_{\mc{X}/M}\dot{\phi}_t(\sqrt{-1}\p\b{\p}\phi_t)^{n+1},\quad \text{mod}\quad \text{Im}\p+\text{Im}\b{\p}.
\end{align*}

Therefore, one gets the following first variation of the Donaldson functional $\mc{L}(\cdot,\psi)$,
\begin{align}\label{first variation}
  \begin{split}
    -\frac{d}{d t}\mc{L}(\phi_t,\psi) &=\int_{\mc{X}}\left(\frac{1}{n+1}\dot{\phi}_t(\sqrt{-1}\p\b{\p}\phi_t)^{n+1}-\frac{\lambda}{m}\dot{\phi}_t(\sqrt{-1}\p\b{\p}\phi_t)^n\wedge\omega
    \right)\frac{\omega^{m-1}}{(m-1)!}\\
    &=\int_{\mc{X}}\dot{\phi}_t(c(\phi_t)-\frac{\lambda}{m}\omega)\wedge (\sqrt{-1}\p\b{\p}\phi_t)^n\wedge\frac{\omega^{m-1}}{(m-1)!}\\
    &=\int_{\mc{X}}\dot{\phi}_t(tr_{\omega}c(\phi_t)-\lambda)(\sqrt{-1}\p\b{\p}\phi_t)^n\wedge\frac{\omega^m}{m!}.
  \end{split}
\end{align}

\begin{prop}\label{prop2}
  $\phi\in F^+(L)$ is a geodesic-Einstein metric if and only if $\phi$ is a critical point of $\mc{L}(\cdot,\psi)$ on $F^+(L)$.
\end{prop}
\begin{proof}
  One direction is easy. If  $\phi\in F^+(L)$ is a geodesic-Einstein metric, then for any smooth curve $\phi_t\in F^+(L)$ with $\phi_0=\phi$, one has by (\ref{first variation}),
  $$\frac{d}{d t}|_{t=0}\mc{L}(\phi_t,\psi)=0.$$
  So $\phi$ is a critical point of  $\mc{L}(\cdot,\psi)$ on $F^+(L)$.

  Conversely, if  $\phi$ is a critical point of $\mc{L}(\cdot,\psi)$, then by taking the following variation
  $$\dot{\phi}_t=tr_{\omega}c(\phi_t)-\lambda,$$
  one gets
  $$0=\int_{\mc{X}}(tr_{\omega}c(\phi_t)-\lambda)^2(\sqrt{-1}\p\b{\p}\phi)^n\wedge\frac{\omega^m}{m!},$$
  and so
  $$tr_{\omega}c(\phi)-\lambda=0,$$
  i.e. $\phi\in F^+(L)$ is a geodesic-Einstein metric.
\end{proof}
\bigskip
For the space of admissible metrics $F^+(L)$, one may consider the following geodesic equation
\begin{align}\label{11}
  c_t(\phi_t):=\ddot{\phi}_t-|\p^V\dot{\phi}_t|^2_{\phi_t}=0.
\end{align}
We view $t$ as a complex parameter and then $\phi_t$ in (\ref{11}) does not depend on the imaginary part of $t$. Moreover, we have
\begin{align}\label{geo}
  (\sqrt{-1}\p\b{\p}_{\mc{X}/M,t}\phi)^{n+1}=(n+1)(\ddot{\phi}_t-|\p^V\dot{\phi}_t|^2_{\phi_t})\sqrt{-1}dt\wedge d\b{t}\wedge (\p\b{\p}_{\mc{X}/M}\phi)^n=0.
\end{align}

Following X. Chen's method (cf. \cite{Chen1}),  one shows easily that there exists a $C^{1,1}$-solution for the geodesic curve equation (\ref{geo}), and moreover, this $C^{1,1}$-solution can be approximated by a family of smooth solutions of the following  equation
\begin{align}\label{family equation}
	(\sqrt{-1}\p\b{\p}_{\mc{X}/M,t}\phi)^{n+1}=\epsilon(n+1)(\sqrt{-1}\p\b{\p}_{\mc{X}/M}\psi)^n\sqrt{-1}dt\wedge d\b{t}
\end{align}
for any $\epsilon>0$ and a fixed metric $\psi\in F^+(L)$. Furthermore, the following analogue of Lemma 7 in \cite{Chen1} holds:
\begin{lemma}[X. Chen, \cite{Chen1}, Lemma 7]\label{approximation}
 The equation (\ref{family equation}) has a smooth solution $\phi_{t,\epsilon}\in F^+(L)$ for any small $\epsilon>0$ and any two given initial metrics $\phi_0,\phi_1\in F^+(L)$. Moreover, $\phi_{t,\epsilon}$ converges uniformly to a $C^{1,1}$ solution $\phi_t$ of the equation (\ref{geo}) as $\epsilon\to 0$, and $\phi_{t,\epsilon}$ has a uniformly bound, i.e. there is a constant $C$ independent of $t$ and $\epsilon$ such that $|\phi_{t,\epsilon}-\psi|<C$.
\end{lemma}

For a family of smooth metrics $\phi_t\in F^+(L)$, one has
\begin{align}\label{Don}
\begin{split}
\mc{L}(\phi_t,\psi)=&\int_{\mc X}\left(\frac{\lambda}{n+1}(\phi_t-\psi)\sum_{k=0}^n(\sqrt{-1}\p\b{\p}\phi_t)^k\wedge(\sqrt{-1}\p\b{\p}\psi)^{n-k}\wedge \frac{\omega^m}{m!}\right.\\
&+\left.\frac{1}{(n+1)(n+2)}(\phi_t-\psi)\sum_{k=0}^{n+1}(\sqrt{-1}\p\b{\p}\phi_t)^k\wedge(\sqrt{-1}\p\b{\p}\psi)^{n+1-k}\wedge\frac{\omega^{m-1}}{(m-1)!}
\right),
\end{split}
\end{align}
and so
\begin{align}\label{Don1}
    \sqrt{-1}\p\b{\p}_t\mc{L}(\phi_t,\psi)=&\int_{\mc X}\left(\frac{\lambda}{n+1}(\sqrt{-1}\p\b{\p}\phi_t)^{n+1}\wedge\frac{\omega^m}{m!}
    -\frac{1}{(n+1)(n+2)}(\sqrt{-1}\p\b{\p}\phi_t)^{n+2}\wedge \frac{\omega^{m-1}}{(m-1)!}\right)\\\nonumber
=&\int_{\mc X}\left(\left(|\p^H\dot{\phi}_t|^2_{\omega}-(\ddot{\phi}_{t\b{t}}-|\p^V\dot{\phi}_t|^2_{\phi_t})(tr_{\omega}c(\phi_t)-\lambda)\right)
(\sqrt{-1}\p\b{\p}\phi_t)^n\frac{\omega^m}{m!}\right)\sqrt{-1}dt\wedge d\b{t}.
\end{align}
Since $\phi_t$ is independent of the imaginary part of $t$, one gets
\begin{align}\label{inequ}
\begin{split}
\frac{d^2}{dt^2}\mc{L}(\phi_t,\psi)&=\int_{\mc{X}}\left(|\p^H\dot{\phi}_t|^2_{\omega}-(\ddot{\phi}_{t\b{t}}-|\p^V\dot{\phi}_t|^2_{\phi_t})(tr_{\omega}c(\phi_t)-\lambda)\right)(\sqrt{-1}\p\b{\p}\phi_t)^n\frac{\omega^m}{m!}\\
&\geq 	-\int_{\mc{X}}(\ddot{\phi}_{t}-|\p^V\dot{\phi}_t|^2_{\phi_t})(tr_{\omega}c(\phi_t)-\lambda)(\sqrt{-1}\p\b{\p}\phi_t)^n\frac{\omega^m}{m!}.
\end{split}
\end{align}

\begin{thm}\label{thm6} For any fixed metric $\psi\in F^+(L)$, the functional $\mc{L}(\cdot,\psi)$ attains the absolute minimum at geodesic-Einstein metrics.
\end{thm}
\begin{proof} Assume $\phi_0$ is a geodesic-Einstein metric on $L$. For any $\phi_1\in F^+(L)$ and any $\epsilon>0$, by Lemma \ref{approximation}, one can connect $\phi_0$ and $\phi_1$ by a path of solutions $\phi_{t,\epsilon}$ of  the equation (\ref{family equation}). From (\ref{inequ}), one has
\begin{align}\label{118}
	\frac{d^2}{dt^2}\mc{L}(\phi_t,\psi)&\geq -\epsilon\int_{\mc{X}}(tr_{\omega}c(\phi_{t,\epsilon})-\lambda)(\sqrt{-1}\p\b{\p}\psi)^n\frac{\omega^m}{m!}.
\end{align}

Let $\{\rho_{A}\}$ be a partition of unity subordinate to some open covering $\{U_A\}$ of $\mc{X}$, where $1\leq A\leq N$. Let $\eta$ be a local smooth function satisfying
$\eta dV=(\sqrt{-1}\p\b{\p}\psi)^n\frac{\omega^m}{m!}$,
where $dV=(\sqrt{-1}dz^1\wedge d\b{z}^1)\wedge\cdots\wedge(\sqrt{-1}dz^m\wedge d\b{z}^m)\wedge (\sqrt{-1}dv^1\wedge d\b{v}^1)\wedge\cdots\wedge(\sqrt{-1}dv^n\wedge d\b{v}^n)$.
Then
\begin{align*}
\int_{\mc{X}}tr_{\omega}c(\phi_{t,\epsilon})(\sqrt{-1}\p\b{\p}\psi)^n\frac{\omega^m}{n!}&=\sum_{A=1}^N\int_{U_A}\rho_A tr_{\omega}c(\phi_{t,\epsilon})\eta dV\\
&=\sum_{A=1}^N\int_{U_A}g^{\alpha\b{\beta}}\left((\phi_{t,\epsilon})_{\alpha\b{\beta}}-(\phi_{t,\epsilon})_{\alpha\b{j}}(\phi_{t,\epsilon})^{i\b{j}}(\phi_{t,\epsilon})_{i\b{\beta}}\right)\rho_A \eta dV\\
&\leq \sum_{A=1}^N\int_{U_A}g^{\alpha\b{\beta}}(\phi_{t,\epsilon})_{\alpha\b{\beta}}\rho_A \eta dV
=\sum_{A=1}^N \int_{U_A}\phi_{t,\epsilon} \p_{\alpha}\p_{\b{\beta}}(g^{\alpha\b{\beta}}\rho_A\eta)dV.
\end{align*}
Since $\phi_{t,\epsilon}$ has a uniform bound, there exists a uniform constant $C_0>0$ such that
$$\int_{\mc{X}}tr_{\omega}c(\phi_{t,\epsilon})(\sqrt{-1}\p\b{\p}\psi)^n\frac{\omega^m}{n!}\leq C_0.$$
Thus by (\ref{118}), one gets
$$\frac{d^2}{dt^2}\mc{L}(\phi_{t,\epsilon},\psi)\geq -\epsilon\left(C_0-\lambda((2\pi c_1(L)^n)[\omega]^m)[\mc{X}]/(m!)\right)\geq -\epsilon C_1$$
for some $C_1>0$. Then by the Taylor expansion of $\mc{L}(\phi_t,\psi)$, one obtains that
\begin{align}\label{Doneqpsilon}
\begin{split}
\mc{L}(\phi_1,\psi)&=\mc{L}(\phi_0,\psi)+\frac{d}{dt}|_{t=0}\mc{L}(\phi_{t,\epsilon},\psi)+\frac{1}{2!}(\frac{d^2}{dt^2}|_{t=\xi}\mc{L}(\phi_{t,\epsilon},\psi))\\
&\geq \mc{L}(\phi_0,\psi)+\frac{d}{dt}|_{t=0}\mc{L}(\phi_{t,\epsilon},\psi)-\epsilon \frac{C_1}{2}
\end{split}
\end{align}
for some $\xi\in [0,1]$. Since $\phi_0$ is a geodesic-Einstein metric on $L$, so $\frac{d}{dt}|_{t=0}\mc{L}(\phi_t,\psi)=0$ for any smooth curve $\phi_t$. Taking $\epsilon\to 0$ to the both sides of (\ref{Doneqpsilon}), one has
\begin{align}
\mc{L}(\phi_1,\psi)\geq \mc{L}(\phi_0,\psi).	
\end{align}
 \end{proof}

Let $M$ be a projective manifold with an ample line bundle $H$. There exists a Hermitian metric $\|\cdot\|_H$ on $H$ such that first Chern form $c_1(H, \|\cdot\|_H)$ of $H$ is positive.
Denote $\omega_H=c_1(H, \|\cdot\|_H)$, which is a K\"{a}hler metric on $M$. For any closed hypersurface $V$ in $M$ defined by some holomorphic section $s$ of $H$, we have
 \begin{prop}\label{prop3}
  Let $\pi:\mc{X}\to M$ be a holomorphic fibration with $\dim M\geq 2$, $L\to \mc{X}$ a relative ample line bundle. Then with respect to $\omega_H$, the Donaldson type functional $\mc{L}_M(\phi,\psi)$ and $\mc{L}_V(\phi,\psi)$ verifies that
  \begin{align}
  \mc{L}_M(\phi,\psi)\geq \frac{1}{m-1}\mc{L}_V(\phi,\psi)-C\max_{p\in M}\int_{\mc{X}_p}(tr_{\omega_H}c(\phi)-\lambda)^2(\sqrt{-1}\p\b{\p}\phi)^n-C_1
\end{align}
for some constants $C, C_1$. Here we considered the restrictive fibration $\mc{X}\to V(\hookrightarrow M)$, the $\mc{L}_V(\phi,\psi)$ is given by 
$$\mc{L}_V(\phi,\psi)=\frac{1}{(m-2)!}\int_V \left(\frac{\lambda}{m}\mc{E}(\phi,\psi)\wedge\omega_H-\frac{1}{n+1}\mc{E}_1(\phi,\psi)\right)\wedge \omega_H^{m-2}.$$
\end{prop}
 \begin{proof} Note that by Lelong-Poincar\'{e} theorem (cf. \cite{Gri1}), one has
\begin{align*}
\frac{\sqrt{-1}}{2\pi}\b{\p}\p\log \|s\|^2_H &=\frac{\sqrt{-1}}{2\pi}\b{\p}\p\log |s|^2+c_1(H, \|\cdot\|_H)\\
&=-Div(s)+c_1(H, \|\cdot\|_H).
\end{align*}
We may and will assume that $\|s\|_H\leq 1$, which can be done by scaling the metric $\|\cdot\|_H$ suitably. Thus
$$\frac{\sqrt{-1}}{2\pi}\int_M \log \|s\|^2_H\wedge \p\b{\p}\eta=\int_V\eta-\int_{M}c_1(H, \|\cdot\|_H)\wedge \eta$$
for any $(n-1,n-1)$ form $\eta$. Thus, for any $k\geq 1$, one has
\begin{align}\label{MV}
\begin{split}
([\omega_H]^k c_1(L)^{m+n-k})[\mc{X}]&=([\omega_H]^k \pi_*(c_1(L)^{m+n-k}))[M]\\
&=([\omega_H]^{k-1} \pi_*(c_1(L)^{m+n-k}))[V]\\
&=([\omega_H]^{k-1} c_1(L)^{m+n-k})[\mc{X}|_V],
\end{split}
\end{align}
where $\mc{X}|_V:=\pi^{-1}(V)$. From (\ref{MV}) and (\ref{lamda}), one knows that
$\frac{\lambda}{m}=\frac{\lambda_{\mc{X}|_V}}{m-1}$, where $\lambda_{\mc{X}|_V}$ is defined by (\ref{lamda}) associated to the holomorphic fibration $\mc{X}|_V\to V$.

Therefore, with respect to the K\"{a}hler metric $\omega_H$, one has
\begin{align}
\label{0.2}
\begin{split}
  (m-1)!&\mc{L}_M(\phi,\psi)=\int_M \left(\frac{\lambda}{m}\mc{E}(\phi,\psi)\wedge\omega_H-\frac{1}{n+1}\mc{E}_1(\phi,\psi)\right)\wedge \omega_H^{m-1}\\
  &=\int_V \left(\frac{\lambda}{m}\mc{E}(\phi,\psi)\wedge\omega_H-\frac{1}{n+1}\mc{E}_1(\phi,\psi)\right)\wedge \omega_H^{m-2}\\
  &\quad -\frac{\sqrt{-1}}{2\pi}\int_M \log \|s\|^2_{H}\p\b{\p}\left(\frac{\lambda}{m}\mc{E}(\phi,\psi)\wedge\omega_H-\frac{1}{n+1}\mc{E}_1(\phi,\psi)\right)\wedge \omega_H^{m-2}\\
  &=(m-2)!\mc{L}_V(\phi,\psi)-\frac{\sqrt{-1}}{2\pi}\int_M \log \|s\|^2_{H}\p\b{\p}\left(\frac{\lambda}{m}\mc{E}(\phi,\psi)\wedge\omega_H-\frac{1}{n+1}\mc{E}_1(\phi,\psi)\right)\wedge \omega_H^{m-2}.
  \end{split}
\end{align}
Note that
 $$\sqrt{-1}\p\b{\p}\mc{E}(\phi,\psi)=\frac{1}{n+1}\left(\pi_*((\sqrt{-1}\p\b{\p}\phi)^{n+1})-\pi_*((\sqrt{-1}\p\b{\p}\psi)^{n+1})\right)$$
and
$$\sqrt{-1}\p\b{\p}\mc{E}_1(\phi,\psi)=\frac{1}{n+2}\left(\pi_*((\sqrt{-1}\p\b{\p}\phi)^{n+2})-\pi_*((\sqrt{-1}\p\b{\p}\psi)^{n+2})\right).$$
Therefore, the second term on the right hand side of (\ref{0.2}) amounts to
\begin{align}\label{0.3}
\begin{split}
  &\int_M \log \|s\|^2_{H}\sqrt{-1}\p\b{\p}\left(\frac{\lambda}{m}\mc{E}(\phi,\psi)\wedge\omega_H-\frac{1}{n+1}\mc{E}_1(\phi,\psi)\right)\wedge \omega_H^{m-2}\\
  &=\frac{1}{n+1}\int_M \log \|s\|^2_H \pi_*\left(\frac{\lambda}{m}(\sqrt{-1}\p\b{\p}\phi)^{n+1}\wedge \omega_H^{m-1}-\frac{1}{n+2}(\sqrt{-1}\p\b{\p}\phi)^{n+2}\wedge \omega_H^{m-2}\right)\\
  &\quad -\frac{1}{n+1}\int_M \log \|s\|^2_H \pi_*\left(\frac{\lambda}{m}(\sqrt{-1}\p\b{\p}\psi)^{n+1}\wedge \omega_H^{m-1}-\frac{1}{n+2}(\sqrt{-1}\p\b{\p}\psi)^{n+2}\wedge \omega_H^{m-2}\right).
  \end{split}
\end{align}
Since
$$\pi_*\left(\frac{\lambda}{m}(\sqrt{-1}\p\b{\p}\phi)^{n+1}\wedge \omega_H^{m-1}\right)=(n+1)\frac{\lambda}{m^2} tr_{\omega_H}c(\phi) (\sqrt{-1}\p\b{\p}\phi)^n\wedge\omega_H^m$$
and
$$\pi_*\left(\frac{1}{n+2}(\sqrt{-1}\p\b{\p}\phi)^{n+2}\wedge \omega_H^{m-2}\right)=\frac{n+1}{2m(m-1)}((tr_{\omega_H}c(\phi))^2-|c(\phi)|^2_{\omega_H})(\sqrt{-1}\p\b{\p}\phi)^n\wedge\omega_H^m,$$
one gets
\begin{align}\label{0.4}
\begin{split}
  &\frac{1}{n+1}\pi_*\left(\frac{\lambda}{m}(\sqrt{-1}\p\b{\p}\phi)^{n+1}\wedge \omega_H^{m-1}-\frac{1}{n+2}(\sqrt{-1}\p\b{\p}\phi)^{n+2}\wedge \omega_H^{m-2}\right)/\omega_H^m\\
  &=\pi_*\left((\frac{\lambda}{m^2} tr_{\omega_H}c(\phi)-\frac{1}{2m(m-1)}((tr_{\omega_H}c(\phi))^2-|c(\phi)|^2_{\omega_H}))(\sqrt{-1}\p\b{\p}\phi)^n\right)\\
  &=\pi_*\left(\left(-\frac{1}{2m^2}(tr_{\omega_H}c(\phi)-\lambda)^2+\frac{1}{2m^2(m-1)}(m|c(\phi)|^2_{\omega_H}-(tr_{\omega_H}c(\phi))^2)+\frac{\lambda^2}{2m^2}\right)(\sqrt{-1}\p\b{\p}\phi)^n\right)\\
  &\geq -\frac{1}{2m^2}\pi_*\left((tr_{\omega_H}c(\phi)-\lambda)^2(\sqrt{-1}\p\b{\p}\phi)^n\right),
  \end{split}
\end{align}
where the last inequality holds by the mean inequality.
From  (\ref{0.3}) and (\ref{0.4}), one has
\begin{align*}
   &\quad\int_M \log \|s\|^2_{H}\sqrt{-1}\p\b{\p}\left(\frac{\lambda}{m}\mc{E}(\phi,\psi)\wedge\omega_H-\frac{1}{n+1}\mc{E}_1(\phi,\psi)\right)\wedge \omega_H^{m-2}\\
     &\leq \left(\int_M(-\frac{1}{2m^2}\log \|s\|^2_H)\omega_H^m\right) \max_{p\in M}\int_{\mc{X}_p}(tr_{\omega_H}c(\phi)-\lambda)^2(\sqrt{-1}\p\b{\p}\phi)^n\\
     &\quad+\frac{1}{n+1}\int_M -\log \|s\|^2_H \pi_*\left(\frac{\lambda}{m}(\sqrt{-1}\p\b{\p}\psi)^{n+1}\wedge \omega_H^{m-1}-\frac{1}{n+2}(\sqrt{-1}\p\b{\p}\psi)^{n+2}\wedge \omega_H^{m-2}\right).
\end{align*}
By (\ref{0.2}), it follows that
\begin{align}\label{rela}
  \mc{L}_M(\phi,\psi)\geq \frac{1}{m-1}\mc{L}_V(\phi,\psi)-C\max_{p\in M}\int_{\mc{X}_p}(tr_{\omega_H}c(\phi)-\lambda)^2(\sqrt{-1}\p\b{\p}\phi)^n-C_1,
\end{align}
where
$$C=\frac{1}{4\pi m(m!)}\int_M(-\log\|s\|^2_H)\omega^m_H$$ and
$$C_1=\frac{1}{2\pi(n+1)((m-1)!)}\int_M -\log \|s\|^2_H \pi_*\left(\frac{\lambda}{m}(\sqrt{-1}\p\b{\p}\psi)^{n+1}\wedge \omega_H^{m-1}-\frac{1}{n+2}(\sqrt{-1}\p\b{\p}\psi)^{n+2}\wedge \omega_H^{m-2}\right).$$
\end{proof}

\section{Geodesic-Einstein metrics and notions of stabilities}

In this section, we introduce certain stabilities of a triple $({\mc{X}},M,L)$ and discuss the relationships between geodesic-Einstein metrics and these stabilities.

A fibration $\mc{Y}\to M-S$, with $S$ a closed subvariety in $M$ of ${\rm codim} S\geq 2$, is called a sub-fibration of the holomorphic fibration ${\mc X}\to M$ if for any $p\in M-S$, the fiber ${\mc Y}_p$ is a complex closed submanifold of the fiber ${\mc X}_p$. Let $\mc{F}$ be the set of sub-fibrations of the holomorphic fibration ${\mc X}\to M$.
Set for any $\mc{Y}\in \mathscr{F}$,
\begin{align}\label{l}
\lambda_{\mc{Y},L}=\frac{2\pi m}{{\dim\mc{Y}/M}+1}\frac{([\omega]^{m-1}c_1(L)^{{\dim\mc{Y}/M}+1})[\mc{Y}]}{([\omega]^mc_1(L)^{\dim\mc{Y}/M})[\mc{Y}]}.
\end{align}

Note that $\lambda_{\mc{Y},L}$ is well defined and independent of the metrics on $L$ by the Stoke's theorem and ${\rm codim} S\geq 2$.

Now we introduce some notions of the stability of a triple $({\mc{X}},M,L)$.
\begin{defn}\label{stability}
    A triple $(\mc{X},M,L)$ is called nonlinear
  semistable (resp. nonlinear stable) if $\lambda_{\mc{Y},L}\geq \lambda_{\mc{X},L}$  (resp. $\lambda_{\mc{Y},L}>\lambda_{\mc{X},L}$) for any sub-fibration $\mc{Y}\in \mathscr{F}$ with
  $\dim\mc{Y}<\dim\mc{X}.$
\end{defn}

We have the following theorem.
\begin{thm}\label{thm3}
  If $L$ admits a geodesic-Einstein metric, then the triple $(\mc{X},M,L)$ is nonlinear semistable.
\end{thm}
\begin{proof}
  For any sub-fibration $\mc{Y}\in\mathscr{F}$ with $\dim \mc{Y}/M=n'$. We make the following convention for indices:
  $$1\leq a,b,c,d\leq n,\quad 1\leq i,j,k,l\leq n',$$
 and use the local admissible coordinate systems $(v^1,\cdots,v^{n'},\cdots,v^n)$ on the fibres of $\mc X$ such that $(v^1,\cdots,v^{n'},0,\cdots,0)\in{\mc Y}$. It  is clear that such
 local admissible coordinate systems are always exist.

  For any $\phi\in F^+(L)$, let $\phi|_{\mc{Y}}$ be the restriction of $\phi$ on the line bundle $L|_{\mc{Y}}$ over the sub-fibration $\mc{Y}\to M-S$. Let
 $c(\phi|_{\mc{Y}})$ denote the geodesic curvature of $\phi|_{\mc{Y}}$ which is defined similarly to (\ref{cphi}). By Lemma \ref{lemma3}, which will be proved later, one has
  \begin{align}
  \begin{split}
    c(\phi)|_{\mc{Y}}-c(\phi|_{\mc{Y}})&=\left((\phi_{\alpha\b{\beta}}-\phi_{\alpha\b{b}}\phi^{\b{b}a}\phi_{a\b{\beta}})-
   (\phi_{\alpha\b{\beta}}-\phi_{\alpha\b{j}}\phi^{\b{j}k}\phi_{k\b{\beta}})\right)\sqrt{-1}dz^{\alpha}\wedge d\b{z}^{\beta}\\
    &=(\phi_{\alpha\b{j}}\phi^{\b{j}k}\phi_{k\b{\beta}}-\phi_{\alpha\b{b}}\phi^{\b{b}a}\phi_{a\b{\beta}})\sqrt{-1}dz^{\alpha}\wedge d\b{z}^{\beta}\leq 0,
    \end{split}
  \end{align}
  where the last inequality holds by setting $B=(\phi_{i\b{j}}), A_2=(\phi_{\alpha\b{j}}), D=(\phi_{A\b{B}}), E=(\phi_{\alpha\b{B}})$ in the following Lemma \ref{lemma3}.
 Thus, when $\phi$ is geodesic-Einstein, one gets
 \begin{align}\label{1111}
 \begin{split}
   \lambda_{\mc{Y},L}&=\frac{2\pi m}{n'+1}\frac{([\omega]^{m-1}c_1(L)^{n'+1})[\mc{Y}]}{([\omega]^{m}c_1(L)^{n'})[\mc{Y}]}\\
   &=\frac{m}{n'+1}\frac{(\omega^{m-1}(\sqrt{-1}\p\b{\p}(\phi|_{\mc{Y}}))^{n'+1})[\mc{Y}]}{(\omega^{m}(\sqrt{-1}\p\b{\p}(\phi|_{\mc{Y}}))^{n'})[\mc{Y}]}\\
   &=\frac{(tr_{\omega}c(\phi|_{\mc{Y}})\omega^m(\sqrt{-1}\p\b{\p}(\phi|_{\mc{Y}}))^{n'})[\mc{Y}]}{(\omega^m(\sqrt{-1}\p\b{\p}(\phi|_{\mc{Y}}))^{n'})[\mc{Y}]}\\
      &\geq \frac{(tr_{\omega}c(\phi)\omega^m(\sqrt{-1}\p\b{\p}(\phi|_{\mc{Y}}))^{n'})[\mc{Y}]}{(\omega^m(\sqrt{-1}\p\b{\p}(\phi|_{\mc{Y}}))^{n'})[\mc{Y}]}\\
   &=tr_{\omega}c(\phi)=\lambda_{\mc{X},L}.
   \end{split}
 \end{align}
  The proof is complete.
\end{proof}
\begin{rem}\label{r1} Note that from the inequalities in (\ref{1111}), we find that if $\phi$ is a geodesic-Einstein metric on $L$, then for any ${\mc{Y}}\in \mathscr{F}$ with $\lambda_{\mc{Y},L}=\lambda_{\mc{X},L}$, the restriction metric $\phi|_{\mc{Y}}$ satisfies $c(\phi|_{\mc{Y}})=c(\phi)|_{\mc{Y}}$. Moreover,
$\phi|_{\mc{Y}}$ is also geodesic-Einstein on $L|_{\mc{Y}}$, that is, $tr_{\omega}c(\phi|_{\mc{Y}})=\lambda_{{\mc Y},L}$.
\end{rem}
\begin{lemma}\label{lemma3}
  Consider the following block matrix with $A=\b{A}^{\top}$,
\begin{equation*}
    A:=\begin{bmatrix}
    \begin{BMAT}{c.cc}{c.cc}
    A_1 & A_2 & A_3\\
    \b{A}_2^{\top}& B_2&B_3\\
    \b{A}_3^{\top}&\b{B}_3^{\top}&C_3
    \end{BMAT}
    \end{bmatrix}=\begin{bmatrix}
    \begin{BMAT}{c.c}{c.c}
    A_1 & E\\
   \b{E}^{\top}&D
    \end{BMAT}
    \end{bmatrix},
 \end{equation*}
where $D$ is a positive definite matrix, then $ED^{-1}\b{E}^{\top}-A_2B_2^{-1}\b{A}_2^{\top}$ is a semipositive definite matrix.
\end{lemma}
\begin{proof}
Since $D$ is a positive definite matrix, so $B_2$ is also a positive definite matrix.  By a direct computation, one has
\begin{equation*}
\begin{bmatrix}
    \begin{BMAT}{c.c}{c.c}
    I & 0 \\
    -\b{B}^{\top}_3B_2^{-1}& I
    \end{BMAT}
    \end{bmatrix} D \begin{bmatrix}\begin{BMAT}{c.c}{c.c}
    I & -B_2^{-1}B_3 \\
    0& I
    \end{BMAT}
    \end{bmatrix}
    =\begin{bmatrix}
    \begin{BMAT}{c.c}{c.c}
    B_2 & 0 \\
    0& C_3-\b{B}^{\top}_3B_2^{-1}B_3
    \end{BMAT}
    \end{bmatrix}.
\end{equation*}
So  $C_3-\b{B}^{\top}_3B_2^{-1}B_3$ is a positive definite matrix and
\begin{equation*}
D^{-1}=\begin{bmatrix}\begin{BMAT}{c.c}{c.c}
    I & -B_2^{-1}B_3 \\
    0& I
    \end{BMAT}
    \end{bmatrix}\begin{bmatrix}
    \begin{BMAT}{c.c}{c.c}
    B_2^{-1} & 0 \\
    0& (C_3-\b{B}^{\top}_3B_2^{-1}B_3)^{-1}
    \end{BMAT}
    \end{bmatrix}\begin{bmatrix}
    \begin{BMAT}{c.c}{c.c}
    I & 0 \\
    -\b{B}^{\top}_3B_2^{-1}& I
    \end{BMAT}
    \end{bmatrix}.
\end{equation*}
Therefore,
\begin{align*}
\begin{split}
	ED^{-1}\b{E}^{\top}&= [A_2,A_3]\begin{bmatrix}\begin{BMAT}{c.c}{c.c}
    I & -B_2^{-1}B_3 \\
    0& I
    \end{BMAT}
    \end{bmatrix}\begin{bmatrix}
    \begin{BMAT}{c.c}{c.c}
    B_2^{-1} & 0 \\
    0& (C_3-\b{B}^{\top}_3B_2^{-1}B_3)^{-1}
    \end{BMAT}
    \end{bmatrix}\begin{bmatrix}
    \begin{BMAT}{c.c}{c.c}
    I & 0 \\
    -\b{B}^{\top}_3B_2^{-1}& I
    \end{BMAT}\end{bmatrix}[\b{A}_2,\b{A}_3]^{\top}\\
    &=A_2B_2^{-1}\b{A}_2^{\top}+(A_3-A_2B_2^{-1}B_3)(C_3-\b{B}^{\top}_3B_2^{-1}B_3)^{-1}(\o{A_3-A_2B_2^{-1}B_3})^{T}.
    \end{split}	
\end{align*}
So
\begin{align}
	ED^{-1}\b{E}^{\top}-A_2B_2^{-1}\b{A}_2^{\top}=(A_3-A_2B_2^{-1}B_3)(C_3-\b{B}^{\top}_3B_2^{-1}B_3)^{-1}(\o{A_3-A_2B_2^{-1}B_3})^{T}
\end{align}
is a semipositive definite matrix.
\end{proof}

\begin{defn}\label{poly}
The triple $(\mc{X},M,L)$ is called nonlinear polystable if there exists a filtration
  \begin{align}
    \mc{X}_0:=\mc{X}\supset \mc{X}_1\supset\cdots\supset \mc{X}_{N},
  \end{align}
  where $(\mc{X}_i,S_i,L)$ is a maximal nonlinear semistable sub-fibration of $(\mc{X}_{i-1},S_{i-1},L)$ with $\lambda_{\mc{X}_i,L}=\lambda_{\mc{X}_{i-1},L}$ for $1\leq i\leq N$, and $(\mc{X}_N,M-S_N,L)$ is also nonlinear stable, and moreover,  there exists a family of metrics $\phi_{i-1}\in F^{+}(L|_{{\mc X}_{i-1}})$ such that $c(\phi_{i-1})|_{\mc{X}_i}= c(\phi_{i-1}|_{{\mc X}_i})$.
 \end{defn}

Clearly, a nonlinear stable triple $(\mc{X},M,L)$ must be nonlinear polystable. Moreover, we have
\begin{thm}\label{thm2}
  If $L$ admits a geodesic-Einstein metric, then the triple $(\mc{X},M,L)$ is nonlinear polystable.
\end{thm}
\begin{proof}
 Since $L$ admits a geodesic-Einstein metric, so by Theorem \ref{thm3}, it is nonlinear semistable. If $\lambda_{\mc{Y},L}>\lambda_{\mc{X},L}$ for all $\mc{Y}\in \mathscr{F}$, then the triple $(\mc{X},M,L)$ is nonlinear stable and so it is nonlinear polystable; otherwise,
  there exists a maximal sub-fibration $\mc{X}_1\subset \mc{X}$ such that $\lambda_{\mc{X}_1,L}=\lambda_{\mc{X},L}$. Thus, by Remark \ref{r1}, one has
    $$c(\phi)|_{\mc{X}_1}=c(\phi|_{\mc{X}_1}).$$
  Now by induction, the proof is complete.
\end{proof}

\section{A special fibration: the projective bundles}

The curvature of direct image sheaf has been computed in (\cite{Bern2}, \cite{Bern3}, \cite{Bern4}, \cite{Ke1}).
In this section, we  will first review some works of Berndtsson on $L^2$ metrics and then discuss the relationships of the geodesic-Einstein metrics on a relative ample bundle $L$ over a holomorphic fibration and the $L^2$ metrics on the related direct image bundles. Finally, we study the special holomorphic fibration $\pi: P(E):=(E-\{0\})/\mb{C}^*\to M$ associated to a
holomorphic vector bundle $E\to (M,\omega)$.

\subsection{$L^2$ metrics on direct image bundles.} For any admissible metric $\phi$ on a relative ample line bundle $L$ over a holomorphic fibration $\pi: \mc{X}\to M$, we consider the direct image sheaf $E:=\pi_*(K_{\mc{X}/M}+L)$. Then $E$ is a holomorphic vector bundle. In fact, for any point $p\in M$, taking a local coordinate neighborhood $(U;\{z^{\alpha}\})$ of $p$, then $\phi+\beta\sum_{\alpha=1}^m |z^\alpha|^2$ is a metric on the line bundle $L\to \mc{X}|_U$, whose curvature is 
$$\sqrt{-1}\p\b{\p}\phi+\beta\sqrt{-1}\sum_{\alpha=1}^m dz^{\alpha}\wedge d\b{z}^{\beta}.$$
By taking $\beta$ large enough,  the curvature of $\phi+\beta\sum_{\alpha=1}^m |z^\alpha|^2$ is positive. The same argument as in \cite[\S 4, page 542]{Bern2},  there exists a local holomorphic frame for $E$. So $E$ is a holomorphic vector bundle.


Following Berndtsson (cf. \cite{Bern2}, \cite{Bern3}, \cite{Bern4}), we define the following $L^2$ metric on
the direct image bundle $E:=\pi_*(K_{\mc{X}/M}+L)$: for any $u\in E_{p}\equiv H^0(\mc{X}_p, (L+K_{\mc{X}/M})_p)$, $p\in M$, then we define
\begin{align}\label{L2 metric}
\|u\|^2=\int_{\mc{X}_p}|u|^2e^{-\phi}. 	
\end{align}
Note that $u$ can be written locally as $u=f dv\wedge e$, where $e$ is a local holomorphic frame for $L|_{\mc X}$, and so locally
$$|u|^2e^{-\phi}=(\sqrt{-1})^{n^2}|f|^2 |e|^2dv\wedge d\b{v}=(\sqrt{-1})^{n^2}|f|^2 e^{-\phi}dv\wedge d\b{v},$$
where $dv=dv^1\wedge \cdots\wedge dv^n$ is the fiber volume.

 The following theorem actually was proved by Berndtsson in \cite[Theorem 1.2]{Bern4}, here we will give a proof for reader's convenience.

\begin{thm}[{\cite[Theorem 1.2]{Bern4}}]
\label{thm4} For any $p\in M$ and let $u\in E_{p}$, one has
\begin{align}\label{cur}
\langle \sqrt{-1}\Theta^{E}u,u\rangle=\int_{\pi^{-1}(p)}c(\phi)|u|^2e^{-\phi}+\langle(1+\Box')^{-1}i_{\b{\p}^V\frac{\delta}{\delta z^{\alpha}}}u,i_{\b{\p}^V\frac{\delta}{\delta z^{\beta}}}u\rangle\sqrt{-1}dz^{\alpha}\wedge d\b{z}^{\beta},
\end{align}
where $\Theta^{E}$ denotes the curvature of the Chern connection on $E$ with the $L^2$ metric defined above, here $\Box'=\n'\n'^*+\n'^*\n$ is the Laplacian on $L|_{\pi^{-1}(p)}$-valued forms on $\pi^{-1}(p)$ defined by the $(1,0)$-part of the Chern connection on $L|_{\pi^{-1}(p)}$.
\end{thm}
\begin{proof}
For any local holomorphic section $u$ of $E$,  following Berndtsson,  
\begin{align}\label{5.1}
	\b{\p}\textbf{u}=dz^{\alpha}\wedge \eta_{\alpha},\quad  D'u=(\Pi_{holo} \mu_{\alpha})dz^{\alpha}
\end{align}
where $\p^{\phi}\textbf{u}=dz^{\alpha}\wedge \mu_{\alpha}$, $\p^{\phi}=e^{\phi}\p e^{-\phi}$, $\Pi_{holo}$ is the projection on the space of holomorphic sections.	

By a direct computation, we have
\begin{align}\label{5.2}
\begin{split}
\sqrt{-1}\p\b{\p}\|u\|^2&=\sqrt{-1}\p\b{\p}\pi_*((\sqrt{-1})^{n^2}\textbf{u}\wedge\o{\textbf{u}}e^{-\phi})\\
&=-	\pi_*(\p\b{\p}\phi(\sqrt{-1})^{n^2}\textbf{u}\wedge\o{\textbf{u}}e^{-\phi})+\pi_*((\sqrt{-1})^{n^2}\mu_{\alpha}\wedge\o{\mu_{\beta}} e^{-\phi})\sqrt{-1}dz^{\alpha}\wedge d\b{z}^{\beta}\\
&\quad +\pi_*((\sqrt{-1})^{n^2}\eta_{\alpha}\wedge \o{\eta}_{\beta}e^{-\phi})\sqrt{-1}dz^{\alpha}\wedge d\b{z}^{\beta}.
\end{split}
\end{align}

For any element $u$ of $E$, we can take a canonical representation \textbf{u} of $u$, $\textbf{u}=u'\delta v\otimes e$, where $u'$ is a local smooth function and holomorphic when restricting on each fiber, $\delta v=\delta v^1\wedge \cdots\wedge \delta v^n$ and $e$ is a local frame of $L$. Then 
\begin{align}\textbf{u}=e^{i_N}(u^0)\end{align}
where $N=\phi^{i\b{l}}\phi_{\b{l}\alpha}dz^\alpha\otimes \frac{\p}{\p v^i}$, $e^{i_N}=\sum_{k=0}^{\infty}\frac{i^k_N}{k!}$, $u^0=u'dv\otimes e$. So 
\begin{align}
\b{\p}\textbf{u}=\b{\p} e^{i_N}(u^0)=e^{i_N}(\b{\p} u^0+i_{\b{\p} N}u^0).
\end{align}
Comparing with (\ref{5.1}), one has 
\begin{align}
\begin{split}
	\b{\p}\textbf{u}&=e^{i_N}(i_{\b{\p} N}u^0)=-dz^{\alpha}\wedge e^{i_N}(i_{\b{\p}\frac{\delta}{\delta z^{\alpha}}}u^0), 
\end{split}
\end{align}
so $\eta_{\alpha}=-e^{i_N}(i_{\b{\p}\frac{\delta}{\delta z^{\alpha}}}u^0)$. 

When restricting on each fiber, $\eta_{\alpha}=-i_{\b{\p}^V\frac{\delta}{\delta z^{\alpha}}}u^0$ and so
\begin{align}
-\eta_{\alpha}\wedge \p\b{\p}\phi=i_{\b{\p}^V\frac{\delta}{\delta z^{\alpha}}}u^0\wedge \p\b{\p}\phi=-u^0\wedge \b{\p}^2\phi_{\alpha}=0,
\end{align}
i.e. $\eta$ is primitive on each fiber. It follows that 
\begin{align}\label{5.3}
\pi_*((\sqrt{-1})^{n^2}\eta_{\alpha}\wedge \o{\eta}_{\beta}e^{-\phi})\sqrt{-1}dz^{\alpha}\wedge d\b{z}^{\beta}=-\langle\eta_{\alpha},\eta_{\beta}\rangle,
\end{align}
where $\langle\bullet,\bullet\rangle$ denotes the Hermitian inner induced by the Hermitian line bundle $(L,\phi)$ and $(\pi^{-1}(p),\p\b{\p}\phi|_{\pi^{-1}(p)})$. 

Moreover, if the holomorphic section $u$ of $E$ satisfies $D'u=0$ at the point $p$, that is, $\mu_{\alpha}$ is orthogonal to holomorphic forms. When restricting on the center fiber $\pi^{-1}(p)$, 
\begin{align}
dz^{\alpha}\wedge\b{\p}\mu_{\alpha}=-\b{\p}\p^{\phi}\textbf{u}=\p^{\phi}\b{\p}\textbf{u}=-dz^{\alpha}\wedge \p^{\phi}\eta_{\alpha}=-dz^{\alpha}\wedge \n'\eta_{\alpha},
\end{align}
so $\b{\p}\mu_{\alpha}=-\n'\eta_{\alpha}$. Since $\mu_{\alpha}$ is the $L^2$ minimal solution, so 
\begin{align}
\mu_{\alpha}=-\b{\p}^*(\Box'')^{-1}\n'\eta_{\alpha}.
\end{align}
Then 
\begin{align}\label{5.4}
\begin{split}
\langle\eta_{\alpha},\eta_{\beta}\rangle-\langle \mu_{\alpha},\mu_{\beta}\rangle &=\langle\eta_{\alpha},\eta_{\beta}\rangle-\langle-\b{\p}^*(\Box'')^{-1}\n'\eta_{\alpha},\mu_{\beta}\rangle\\
&=\langle\eta_{\alpha},\eta_{\beta}\rangle-\langle(\Box'+1)^{-1}\n'\eta_{\alpha},\n'\eta_{\beta}\rangle\\
&=\langle\eta_{\alpha},\eta_{\beta}\rangle-\langle(\Box'+1)^{-1}\Box'\eta_{\alpha},\eta_{\beta}\rangle\\
&=\langle(\Box'+1)^{-1}\eta_{\alpha},\eta_{\beta}\rangle.
\end{split}
\end{align}

On the other hand, by Lemma \ref{lemma1}, we have
\begin{align}\label{5.5}
\sqrt{-1}\pi_*(\p\b{\p}\phi(\sqrt{-1})^{n^2}\textbf{u}\wedge\o{\textbf{u}}e^{-\phi})=\int_{\pi^{-1}(p)}c(\phi)|u|^2 e^{-\phi}.
\end{align}

For any element $u$ of $E_p$, one can take a local holomorphic extension of $u$ and such that $D'u=0$ at $p$. From (\ref{5.2}), (\ref{5.3}), (\ref{5.4}) and (\ref{5.5}),  one has at the point $p$, 
\begin{align}
\begin{split}
\langle\sqrt{-1}\Theta^E u,u\rangle &=-\sqrt{-1}\p\b{\p}\|u\|^2\\
&=\int_{\pi^{-1}(p)}c(\phi)|u|^2e^{-\phi}+\langle(1+\Box')^{-1}i_{\b{\p}^V\frac{\delta}{\delta z^{\alpha}}}u,i_{\b{\p}^V\frac{\delta}{\delta z^{\beta}}}u\rangle\sqrt{-1}dz^{\alpha}\wedge d\b{z}^{\beta},
\end{split}
\end{align}
which completes the proof.
\end{proof}

Taking trace to both sides of (\ref{cur}) and by the semipositivity of the second term in the right hand side of (\ref{cur}), one has
\begin{align}\label{cur1}
  g^{\alpha\b{\beta}}\langle \Theta^E_{\alpha\b{\beta}} u,u\rangle \geq \int_{\pi^{-1}(p)} tr_{\omega}c(\phi)|u|^2e^{-\phi}.
\end{align}
From this we have the following proposition.
\begin{prop}\label{thm1} If $L$ admits a geodesic-Einstein metric $\phi$, then one has
\begin{align}
\frac{\deg E}{\text{\rm rank} E}\geq{{\rm Vol_\omega(M)}\over{2\pi m}}\lambda_{\mc{X},L},
\end{align}		
where $\deg E:=(c_1(E)\wedge[\omega]^{m-1})[M]$, ${\rm Vol_\omega(M)}=\omega^m[M]$.
\end{prop}

\begin{proof}
Let $\{u_A\}, 1\leq A\leq \text{rank}E$, be a local holomorphic frame of $E$, and  $h$ denote the $L^2$ metric on $E$. Set
$$h_{A\b{B}}=\langle u_A, u_B\rangle=\int_{\mc{X}_p}u_A \o{u_B}e^{-\phi}.$$
Since $\phi$ is geodesic-Einstein, then from (\ref{cur1}), one has for any $u=\sum_{A=1}^{\text{rank}E}a^A u_A\in E$ ,
  \begin{align*}
    g^{\alpha\b{\beta}}\Theta_{\alpha\b{\beta}A\b{B}}a^A\b{a}^B\geq \lambda_{{\mc X},L}h_{A\b{B}}a^A\b{a}^B,
  \end{align*}
and then by taking trace with respect to the $L^2$ metric $h$, one gets
  \begin{align*}
    g^{\alpha\b{\beta}}R_{\alpha\b{\beta}}:= g^{\alpha\b{\beta}}\Theta_{\alpha\b{\beta}A\b{B}}h^{A\b{B}}\geq \lambda_{{\mc X},L}\cdot\text{rank}E.
  \end{align*}
  Integrating both sides of the above inequality over $M$, one has
  \begin{align*}
    \lambda_{{\mc X},L}\leq \frac{\deg E}{\text{rank} E}\frac{2\pi m}{\rm Vol_\omega(M)},
  \end{align*}
from which the proposition follows.
\end{proof}
From Proposition \ref{thm1}, we have the following corollary.
\begin{cor}\label{ob3}
If $\phi$ is a geodesic-Einstein metric on $L$, then
 \begin{align}\label{ob2}
\frac{\deg E}{\text{rank} E}={{\rm Vol_\omega(M)}\over{2\pi m}}\lambda_{\mc{X},L}
\end{align}
if and only if
 the induced $L^2$ metric (\ref{L2 metric}) is a Hermitian-Einstein metric on $E=\pi_*(L+K_{\mc{X}/M})$.
\end{cor}
\begin{proof}
  If (\ref{ob2}) holds, then all equalities in the proof of Proposition \ref{thm1} hold.  In particular, one has
  \begin{align}
    g^{\alpha\b{\beta}}\Theta_{\alpha\b{\beta}A\b{B}}= \lambda h_{A\b{B}},
  \end{align}
  that is, the induced $L^2$ metric (\ref{L2 metric}) on $E$ is a Hermitian-Einstein metric. The converse part is obvious.
\end{proof}
Now we consider the asymptotic behavior  of the corresponding quantities $\frac{\deg E^{(k)}}{\text{rank}E^{(k)}}$ when replacing $L$ by $L^k$ as $k\to \infty$. Firstly, by Bergman Kernel expansion, one has
\begin{align}\label{rank}
\text{rank}(\pi_*(L^k+K_{\mc{X}/M}))=\dim H^0(\mc{X}_y, L^k+K_{\mc{X}/M}|_y)=b_0 k^n+b_1k^{n-1}+O(k^{n-2}),
\end{align}
where
\begin{align}
	b_0:=\pi_*\left(\frac{c_1(L)^n}{n!}\right),\quad b_1=\pi_*\left(\frac{c_1(L)^{n-1}c_1(K_{\mc{X}/M})}{2(n-1)!}\right).
\end{align}

On the other hand, one has for any complex vector bundle $E$,
$$ch(E)=\text{rank}E+c_1+\frac{1}{2}(c_1^2-2c_2)+\frac{1}{6}(c_1^3-3c_1c_2+c_3)+\cdots$$
and
$$Td(E)=1+\frac{1}{2}c_1+\frac{1}{12}(c_1^2+c_2)+\frac{1}{24}c_1c_2+\cdots.$$
So from the Grothendieck-Riemann-Roch Theorem and Kodaira vanishing theorem, one gets
\begin{align}
\begin{split}
c_1(\pi_*(L^k+K_{\mc{X}/M}))&=c_1(\pi_!(L^k+K_{\mc{X}/M}))\\
&=\left\{\pi_*(ch(L^k+K_{\mc{X}/M})Td(T_{\mc{X}/M}))\right\}^{(1,1)}\\\nonumber
&=\pi_*\left(\frac{c_1(L)^{n+1}}{(n+1)!}\right)k^{n+1}+\pi_*\left(\frac{1}{2}\frac{c_1(L)^n}{n!}c_1(K_{\mc{X}/M})\right)k^n+O(k^{n-1}).
\end{split}
\end{align}
Now by setting $E^{(k)}:=\pi_*(L^k+K_{\mc{X}/M})$, one has
\begin{align}\label{expanding}
\deg E^{(k)}=(c_1(\pi_*(L^k+K_{\mc{X}/M}))[\omega]^{m-1})[M]=a_0k^{n+1}+a_1 k^n+O(k^{n-1}),
\end{align}
where
\begin{align}
a_0=\left(\frac{c_1(L)^{n+1}}{(n+1)!}[\omega]^{m-1}\right)[\mc{X}],\quad a_1=\left(\frac{1}{2}\frac{c_1(L)^n}{n!}c_1(K_{\mc{X}/M})[\omega]^{m-1}\right)[\mc{X}].
\end{align}

Hence, by (\ref{rank}) and (\ref{expanding}), one has
\begin{align}\label{Ek}
  \frac{\deg E^{(k)}}{\text{rank}E^{(k)}}=\frac{1}{n+1}\frac{(c_1(L)^{n+1}[\omega]^{m-1})[\mc{X}]}{c_1(L)^n[\mc{X}/M]}k
  +DF(\mc{X},L)+O(k^{-1}).
\end{align}
where
\begin{align}
DF(\mc{X},L):=\frac{a_1b_0-a_0b_1}{b_0^2},	
\end{align}
which can be viewed as an analogue of the Donaldson-Futaki invariant (cf. \cite{Don4}) on the obstruction for the existence of geodesic-Einstein metrics.

\begin{prop}\label{thm7}
  If $DF(\mc{X},L)<0$, then there exists no geodesic-Einstein metrics on $L$.
\end{prop}
\begin{proof}
  By (\ref{Ek}), if $DF(\mc{X},L)<0$, one has
  $$\frac{\deg E^{(k)}}{\text{rank}E^{(k)}}<\frac{1}{n+1}\frac{(c_1(L)^{n+1}[\omega]^{m-1})[\mc{X}]}{c_1(L)^n[\mc{X}/M]}k
  =\frac{1}{n+1}\frac{(c_1(L^k)^{n+1}[\omega]^{m-1})[\mc{X}]}{c_1(L^k)^n[\mc{X}/M]}=\frac{\text{Vol}_{\omega}(M)}{2\pi m}\lambda_{\mc{X},L^k}$$
  for $k>0$ large enough. Therefore, by Proposition \ref{thm1}, there exists no geodesic-Einstein metric on $L^k$ and so on $L$. In fact, if $\phi$ is a geodesic-Einstein metric on $L$,
   then induced metric $k\phi$ on $L^k$ satisfies
  $$tr_{\omega}c(k\phi)=ktr_{\omega}c(\phi)=k\lambda,$$
  i.e. $k\phi$ is a geodesic-Einstein metric on $L^k$ and this yields a contradiction.
\end{proof}

\subsection{Projective bundles} For any holomorphic vector bundle $E\to M$ of rank $r$, let $\pi:P(E^*)\to M$ be the associated projective fibre bundle to the dual bundle $E^*\to M$, and $\mc{O}_{P(E^*)}(1)$ the hyperplane line bundle over $P(E^*)$. Clearly, the line bundle $\mc{O}_{P(E^*)}(1)\to P(E^*)$ is relative ample over the holomorphic fibration $\pi:P(E^*)\to M$.
Also by Lemma 5.37 in \cite{Shiff}, one knows that
\begin{align}\label{1.4}
  E=\pi_*(\mc{O}_{P(E^*)}(1)).
\end{align}
By Proposition 2.2 in \cite{Kob1}, one has
\begin{align}\label{1.5}
  K_{\mc{X}/M}=K_{P(E^*)/M}=\pi^*(\det E)-r\mc{O}_{P(E^*)}(1).
\end{align}
Set
\begin{align}\label{1.6}
  L=\mc{O}_{P(E^*)}(1)-K_{\mc{X}/M}=(r+1)\mc{O}_{P(E^*)}(1)-\pi^*(\det E),
\end{align}
which is also a relative ample line bundle over $P(E^*)$. Then
\begin{align}\label{direct image bundles}
  E=\pi_*(\mc{O}_{P(E^*)}(1))=\pi_*(L+K_{\mc{X}/M}).
\end{align}

By a direct computation, one has
\begin{align}
\begin{split}
  &\quad\frac{\pi_*(c_1(L)^{r})}{\pi_*(c_1(L)^{r-1})}=\frac{\pi_*(((r+1)c_1(\mc{O}_{P(E^*)}(1))-c_1(\det E))^{r})}{(r+1)^{r-1}\pi_*(c_1(\mc{O}_{P(E^*)}(1))^{r-1})}\\
  &=\frac{(r+1)^{r}\pi_*(c_1(\mc{O}_{P(E^*)}(1))^{r})-r(r+1)^{r-1}\pi_*(c_1(\mc{O}_{P(E^*)}(1))^{r-1})c_1(E)}{(r+1)^{r-1}\pi_*(c_1(\mc{O}_{P(E^*)}(1))^{r-1})}\\
  &=c_1(E).
  \end{split}
\end{align}
From above equality, one obtains that
\begin{align}\label{deg of E}
\begin{split}
	\frac{\deg E}{\text{rank}E}&=\frac{(c_1(L)^r\wedge[\omega]^{m-1})[P(E^*)]}{r(c_1(L)^{r-1})[P(E^*)/M]}\\
	&=\frac{\text{Vol}_{\omega}(M)}{2\pi m}\lambda_{P(E^*),L}.
\end{split}	
\end{align}

For any metric $\phi$ on $\mc{O}_{P(E^*)}(1)$, the first Chern class $c_1(E)$ can be written as (cf. \cite{FLW})
$$c_1(E)=\pi_* c_{1}(\mc{O}_{P(E^*)}(1))^{r}=\left[\left(\frac{1}{2\pi}\right)^r\pi_*(rc(\phi)(\sqrt{-1}\p\b{\p}\phi)^{r-1})\right].$$
Also by a conformal transformation (cf. \cite{Ko3}, Proposition 2.2.23), one can find some metric $\phi_E$ such that
$$c_1({\det E},\phi_E)=\left(\frac{1}{2\pi}\right)^r\pi_*(rc(\phi)(\sqrt{-1}\p\b{\p}\phi)^{r-1}).$$

When $\phi$ is geodesic-Einstein, we have
$$tr_{\omega}c_1(\det E,\phi_E)=\frac{1}{2\pi}rtr_{\omega}c(\phi)=\text{constant}.$$
So from (\ref{1.6}), the induced $ \phi_L=(r+1)\phi-\phi_E$ on $L$
is also geodesic-Einstein. In fact, since $\phi_E$ is independent of fibers, so
\begin{align}\label{geodesic L}
  tr_{\omega}c(\phi_L)=tr_{\omega}c((r+1)\phi-\phi_E)=(r+1)tr_{\omega}c(\phi)-tr_{\omega}c_1(\det E,\phi_E)=\text{constant}.
\end{align}
By (\ref{direct image bundles}), (\ref{deg of E}), (\ref{geodesic L}) and Corollary \ref{ob3}, we have
\begin{prop}\label{ob4}
  If $\mc{O}_{P(E^*)}(1)$ admits a geodesic-Einstein metric, then the induced $L^2$ metric  on $E$ is a Hermitian-Einstein metric.
\end{prop}

For a holomorphic vector bundle  $E\to M$, Kobayashi \cite{Ko1} established a correspondence between the Finsler metrics $G$ on $E$ and the Hermitian metrics $\phi$ on the  hyperplane line bundle
$\mc{O}_{P(E)}(1)\to P(E)$ with $\p\b\p\phi=\p\b\p\log G$ and defined the notion of the Finsler-Einstein metric on $E$ over a K\"{a}hler manifold manifold $M$ (more details, cf. \cite{Ko1}, \cite{FLW}, \cite{FLW1}). As applications, we have

\begin{lemma}\label{lemma2} A strongly pseudo-convex Finsler metric $G$ on $E$ is Finsler-Einstein if and only if the corresponding metric $\phi$ on $\mc{O}_{P(E)}(1)$
is geodesic-Einstein.
\end{lemma}

\begin{proof} For any Finsler metric $G$ on $E$, let $\phi$ be the corresponding Hermitian metric on the line bundle $\mc{O}_{P(E)}(1)$, which is also an admissible metric on $\mc{O}_{P(E)}(1)$  (cf. \cite{Ko1}, \cite{FLW}, \cite{FLW1}). With respect to a holomorphic trivialization of $E\to M$, by the standard procedure, one gets a local homogeneous holomorphic coordinate system $\{[\zeta]|\zeta=(\zeta^1,\cdots,\zeta^{r})\neq 0\}$ on the fibres of $P(E)$.

For a point $p\in P(E)$, without loss of generality, we assume that $\zeta^1(p)\neq 0$ and denote
 $v^{a}=\frac{\zeta^{a}}{\zeta^{1}}$, $2\leq a\leq r$.
Then we have
$$\frac{\p}{\p\zeta^1}=-\frac{\zeta^a}{(\zeta^1)^2}\frac{\p}{\p v^a}\quad \frac{\partial}{\partial \zeta^{a}}=\frac{1}{\zeta^{1}}\frac{\partial}{\partial v^{a}},\quad 2\leq a\leq r.$$
Denote by $((\log G)^{ab})_{2\leq a,b\leq r}$ (resp. $(G^{i\b{j}})_{1\leq i,j\leq r}$) the inverse of the   matrix $\left(\frac{\p^2\log G}{\p v^a\b{v}^b}\right)_{2\leq a,b\leq r}$ (resp. $\left(\frac{\p^2 G}{\p\zeta^i\p\b{\zeta}^j}\right)_{1\leq i,j\leq r}$).
  Then
  $$(\log G)^{\bar{b}a}=\frac{G}{|\zeta^{1}|^{2}}\left(-\frac{\zeta^{a}}{\zeta^{1}}G^{\bar{b}1}+G^{\bar{b}a}
  +\frac{\bar{\zeta}^{b}\zeta^{a}}{|\zeta^{1}|^{2}}G^{\bar{1}1}
  -\frac{\bar{\zeta}^{b}}{\bar{\zeta}^{1}}G^{\bar{1}a}\right).$$
  By a direct computation, one gets
  \begin{align}\label{111}
  c(\phi)=\sqrt{-1}(\phi_{\alpha\b{\beta}}-\phi_{\alpha\b{b}}\phi^{\b{b}a}\phi_{a\b{\beta}})dz^{\alpha}\wedge d\b{z}^{\beta}=-\sqrt{-1} K_{i\b{j}\alpha\b{\beta}}\frac{v^i\b{v}^j}{G}dz^{\alpha}\wedge d\b{z}^{\beta}:=-\Psi,
  \end{align}
where $\Psi:=K_{i\b{j}\alpha\b{\beta}}\frac{v^i\b{v}^j}{G}$ is the Kobayashi curvature of the Finsler metric $G$ (cf. \cite{FLW}, (1.21)). Now  the lemma is follows directly from (\ref{111})
and the definitions of the Finsler-Einstein metric (cf. \cite{FLW}, Definition 3.1) and the geodesic-Einstein metric.
\end{proof}

\begin{thm}\label{Finsler} $E$ admits a Finsler-Einstein metric if and only if $E$ admits a Hermitian-Einstein metric.
\end{thm}
\begin{proof}
 One direction is obvious, since a Hermitian-Einstein metric must be a Finsler-Einstein metric.

 Conversely, if $E$ admits a Finsler-Einstein metric $G$, then by Lemma \ref{lemma2}, the induced metric $\phi$ is a geodesic-Einstein metric on $\mc{O}_{P(E)}(1)$.  So from Proposition \ref{ob4}, the induced $L^2$ metric (\ref{L2 metric}) is a Hermitian-Einstein metric on $E^*$. Therefore, the dual Hermitian metric on $E$ is also a Hermitian-Einstein metric.
\end{proof}

Next, we will discuss some relations between the nonlinear stability of the holomorphic fibration $(P(E), M, \mc{O}_{P(E)}(1))$ and the stability of the holomorphic vector bundle $E$. 

\begin{prop}
\begin{enumerate}
  \item If $(P(E),M,\mc{O}_{P(E)}(1))$ is nonlinear semistable (resp. nonlinear stable), then $E$ is semistable (resp. stable).
  \item If $E$ is polystable, then $(P(E),M,\mc{O}_{P(E)}(1))$ is nonlinear polystable.
\end{enumerate}
\end{prop}
\begin{proof}
\begin{enumerate}
\item For any coherent subsheaf $\mc{F}$ of $\mc{O}(E)$, $F=\mc{F}|_{M-S}$ is actually a holomorphic vector bundle over $M-S$ for some subvariety $S$ of $M$ of ${\rm codim}(S)\geq 2$. So $P(F)\to M-S$ is a subfibration of $P(E)\to M$. Then
\begin{align}\label{semistable inequality 1}
	\begin{split}
		\frac{\deg_{\omega}\mc{F}}{\text{rank}\mc{F}}&=\frac{\int_M c_1(F)\wedge\omega^{m-1}}{\text{rank}F}\\
		&=-\frac{([\omega]^{m-1}c_1(\mc{O}_{P(F)}(1))^{\text{rank}F})[P(F)]}{\text{rank}F}\\
		&=-\frac{\text{Vol}_{\omega}(M)}{2\pi m}\lambda_{P(F),\mc{O}_{P(F)}(1)}.
	\end{split}
\end{align}
If $(P(E),M,\mc{O}_{P(E)}(1))$ is nonlinear semistable (resp. nonlinear stable), then
\begin{align}\label{semistable inequality 2}
	\lambda_{P(F),\mc{O}_{P(F)}(1)}\geq (\text{resp.} >) \lambda_{P(E),\mc{O}_{P(E)}(1)}.
\end{align}
From (\ref{semistable inequality 1}) and (\ref{semistable inequality 2}), one has
\begin{align}
\frac{\deg_{\omega}\mc{F}}{\text{rank}\mc{F}}\leq (\text{resp.} <)\frac{\deg_{\omega}E}{\text{rank}E}.	
\end{align}
So $E$ is semistable (resp. stable).
\item If $E$ is a holomorphic ploystable vector bundle over $M$, by Donaldson-Uhlenbeck-Yau Theorem (cf. \cite{Yau}), there exists a Hermitian-Einstein metric on $E$. From Theorem \ref{Finsler}, Lemma \ref{lemma2} and Theorem \ref{thm2}, one knows that $(P(E),M,\mc{O}_{P(E)}(1))$ is nonlinear polystable.
\end{enumerate}
\end{proof}

\end{document}